\documentclass[a4paper,12pt]{amsart}
\usepackage[utf8]{inputenc}
\usepackage{amsmath}
\usepackage{amssymb}
\usepackage{amsthm}
\usepackage{amsrefs}
\usepackage{commath}
\usepackage{graphicx}
\usepackage{dsfont}

\usepackage{amsmath,amsfonts,epsfig}
\usepackage{type1cm,ae}
\usepackage{graphicx}
\usepackage{xspace}
\usepackage{setspace}
\usepackage[english]{babel}
\usepackage{tikz}
\usetikzlibrary{arrows,patterns}
\usepackage{amssymb,latexsym}
\usetikzlibrary{arrows,patterns}

\usepackage{mathtools}

\usepackage{wrapfig}
\usepackage{capt-of}
\makeatletter

\def\Xint#1{\mathchoice
{\XXint\displaystyle\textstyle{#1}}%
{\XXint\textstyle\scriptstyle{#1}}%
{\XXint\scriptstyle\scriptscriptstyle{#1}}%
{\XXint\scriptscriptstyle\scriptscriptstyle{#1}}%
\!\int}
\def\XXint#1#2#3{{\setbox0=\hbox{$#1{#2#3}{\int}$ }
\vcenter{\hbox{$#2#3$ }}\kern-.6\wd0}}

\def\dashint{\Xint-}
\newcommand{\ind}{\protect\raisebox{2pt}{$\chi$}}

\newcommand{\restr}{%
  \,\raisebox{-.260ex}{\reflectbox{\rotatebox[origin=br]{-90}{$\lnot$}}}\,%
}
\title{Accessible parts of boundary for simply connected domains}

\begin{document}
\author[P. Koskela]{Pekka Koskela}
\address[Pekka Koskela]{Department of Mathematics and Statistics, University of Jyv\"askyl\"a, P.O. Box 35, FI-40014 Jyv\"askyl\"a, Finland}
\email{pekka.j.koskela@jyu.fi}

\author[D. Nandi]{Debanjan Nandi}
\address[Debanjan Nandi]{Department of Mathematics and Statistics, University of Jyv\"askyl\"a, P.O. Box 35, FI-40014 Jyv\"askyl\"a, Finland}
\email{debanjan.s.nandi@jyu.fi}

\author[A. Nicolau]{Artur Nicolau}
\address[Artur Nicolau]{Departament de Matemàtiques, Universitat Autònoma de Barcelona, 08193 Bellaterra . Barcelona, Spain}
\email{artur@mat.uab.cat}

\subjclass[2000]{}
\keywords{simply connected, John domain, Hardy inequality}
\thanks{A. Nicolau was partially supported by the grants 2014SGR75 of Generalitat de 
Catalunya and MTM2014-51824-P and MTM2017-85666-P of Ministerio de Ciencia e Innovaci\'on.  
P. Koskela and D. Nandi were partially supported by the Academy of Finland 
grant 307333.}

\theoremstyle{plain}
\newtheorem{thm}{Theorem}[section]
\newtheorem{lem}[thm]{Lemma}
\newtheorem{prop}[thm]{Proposition}
\newtheorem{cor}[thm]{Corollary}

\theoremstyle{definition}
\newtheorem{defn}[thm]{Definition}
\newtheorem{exm}[thm]{Example}
\newtheorem{prob}[thm]{Problem}

\theoremstyle{remark}
\newtheorem{rem}[thm]{Remark}

%\sectionfont{\fontsize{12}{16}\selectfont}

\begin{abstract}
For a bounded simply connected domain $\Omega\subset\mathbb{R}^2$, any point $z\in\Omega$ and any $0<\alpha<1$, we give a lower bound for the $\alpha$-dimensional Hausdorff content of the set of points in the boundary of $\Omega$  which can be joined to $z$ by a John curve with a suitable John constant depending only on $\alpha$, in terms of the distance of $z$ to $\partial\Omega$. In fact this set in the boundary contains the intersection $\partial\Omega_z\cap\partial\Omega$ of the boundary of a John sub-domain $\Omega_z$ of $\Omega$, centered at $z$, with the boundary of $\Omega$. This may be understood as a quantitative version of a result of Makarov. This estimate is then applied to obtain the pointwise version of a weighted Hardy inequality. 
\end{abstract}

\maketitle

\section{Introduction}Let $\Omega\subset\mathbb{C}$ be a domain. We say that $\Omega$ is $C$-John with center $z_0$ if for any $z\in\Omega$ there exists a rectifiable curve $\gamma_z$ joining $z$ and $z_0$ in $\Omega$ such that for any point $z'$ in the image of $\gamma_z$, it holds that 
$$C\text{d}_{\Omega}(z')\geq l(\gamma_z(z',z)),$$ where $\text{d}_{\Omega}(z'):=\text{dist}(z',\partial\Omega)$ and $l(\gamma_z(z',z))$ is the length of the subcurve between $z'$ and $z$. Given $A\subset\mathbb{C}$, we define the $\alpha$-Hausdorff content as $$\mathcal{H}^{\alpha}_{\infty}(A):=\inf \{\;\sum_{j=1}^{\infty}\text{diam}(E_j)^{\alpha}:\,E_j\subset\mathbb{C},\; A\subset \underset{j\in\mathbb{N}}\cup E_j \}.$$ 

Given a simply connected John domain and $z\in\Omega$ there is a John subdomain $\Omega_z$ with center $z$ so that, for the ball in the intrinsic metric (defined by taking the infimum of the lengths of rectifiable paths in the domain joining pairs of points) of radius $2\text{d}_{\Omega}(z)$, we have $B_{\Omega}(z,2\text{d}_{\Omega}(z)) \subset \Omega_z$; see \cite{KL}, for example. This statement is quantitative in the sense that the John constant of $\Omega_z$ depends only on the John constant of $\Omega$. It is easy to see that this conclusion fails for general simply connected $\Omega$: we may not capture all of $\partial B_{\Omega}(z,2\text{d}_{\Omega}(z))\cap \partial \Omega$ by $\partial \Omega_z$ for a John subdomain $\Omega_z$ for a fixed John constant. The best we can hope for is to capture a part of $\partial\Omega$ of $\mathcal{H}^1$-content of the order of 
$\text{d}_{\Omega}(z)$. Our main result gives a rather optimal conclusion.
 
\begin{thm}\label{main_reg}
Let $\Omega$ be a bounded, simply connected domain in the plane. Let $0<\alpha<1$ be fixed. Given $z\in\Omega$, there is a John subdomain $\Omega_{z}\subset \Omega$ with center $z$ and John constant depending only on $\alpha$ such that 
\begin{equation*}
\mathcal{H}^{\alpha}_{\infty}(\partial \Omega_{z}\cap\partial \Omega)\geq c(\alpha) \text{d}_{\Omega}(z)^\alpha.
\end{equation*}
\end{thm}

The motivation for Theorem \ref{main_reg} partially arises from the weighted pointwise Hardy inequalities (see \cite{haj},  \cite{KM}, \cite{KL})
\begin{equation}\label{pointwise}
 |u(x)|\leq Cd_{\Omega}(x)^{1-\frac{\beta}{p}}\sup_{0<r<2\text{d}_{\Omega}(x)}\left(\dashint_{B(x,r)\cap\Omega}|\nabla u|^q\text{d}_{\Omega}^{q\beta/p}\right)^{1/q}
\end{equation}
 where $u\in C^\infty_0(\Omega)$, $1<q<p$ and $-\infty<\beta<\infty$.
This inequality immediately yields the usual weighted Hardy inequality (see \cite{H1},\cite{H2} for the classical Hardy inequality and \cite{Necas}, \cite{KL} for higher dimensional versions of it)
$$
 \int_{\Omega}|u(x)|^p \text{d}_{\Omega}(x)^{\beta-p}\, dx \leq C\int_{\Omega} |\nabla u(x)|^p \text{d}_{\Omega}(x)^{\beta}\, dx
$$
via the boundedness of the Hardy-Littlewood maximal operator on $L^{p/q}$. The pointwise Hardy inequalities were shown in \cite{KL} to hold for any simply connected John domain for all $1<p<\infty$ and every $\beta<p-1$. This is the optimal range even for Lipschitz domains; see \cite{Necas}. From Theorem \ref{main_reg} together with Theorem 5.1 in \cite{KL} we have the following corollary.

\begin{cor}
Let $\Omega\subset\mathbb{C}$ be simply connected. Let $1<p<\infty$. Then, for 
each $\beta<p-1$ there exist $1<q(\beta,p)<p$ and $C>0$ such that the 
weighted pointwise Hardy 
inequality (\ref{pointwise}) holds for each $x\in\Omega$.
\end{cor}

Above, $q$ and $C$ are independent of $\Omega.$
The corresponding weighted Hardy inequalities were already established in \cite{L}. Our proof of Theorem \ref{main_reg} is based on the following estimate for conformal maps which we expect to be of independent interest. Let $\mathbb{H}$ be the upper half plane.
 \begin{thm}\label{reg_boundary}
Let $f:\mathbb{H}\rightarrow\Omega$ be a conformal map. Let $0<\alpha<1$ be fixed. Then there exists $C(\alpha)>0$ such that the following holds. 

Given $z_0=x_0+iy_0 \in\mathbb{H}$, there exists a set $E=E(z_0,\alpha)\subset (x_0-y_0/2,x_0+y_0/2)$ such that  
\begin{enumerate}
   \item[(a)]${\mathcal{H}}^{\alpha}_{\infty}(E)\geq \frac{y^{\alpha}_0}{C(\alpha)}$
   \item[(b)] $\frac{1}{C(\alpha)}|f'(z_0)|\leq |f'(w)|\leq C(\alpha)|f'(z_0)|$
  \end{enumerate}
  
  for any point $w$ in the sawtooth region $S(E):=\{x+iy:x\in (x_0-y_0/2,x_0+y_0/2),\,d(x,E)\leq y<y_0\}$.
\end{thm}
Theorem \ref{reg_boundary} can be understood as a quantitative version of a result of Makarov; see Theorem 5.1 of \cite{Ma}, see also corollary 1.4 of \cite{R}. Our proof of Theorem \ref{reg_boundary} uses Makarov's idea of approximating Bloch functions by dyadic martingales. Theorem \ref{main_reg} then follows from Theorem \ref{reg_boundary} and Lemma \ref{imagecontentbound} below; in fact we have that $(\Omega_{z_0},\text{d}_{\Omega_{z_0}})$ is bilipschitz equivalent to the sawtooth region in Theorem \ref{reg_boundary}. 

 \section{Preliminaries}
 
 Let $\mathbb{D}$ be the unit disk in the complex plane.
  A function $g:\mathbb{D}\rightarrow\mathbb{C}$ is called a Bloch function if it is analytic and 
  \begin{equation*}
   \|g\|_{\mathcal{B}}:=\underset{z\in\mathbb{D}}\sup (1-|z|^2)|g'(z)|<\infty.
  \end{equation*}
This defines a seminorm. The Bloch functions form a complex Banach space $\mathcal{B}$ with the norm $|g(0)|+\|g\|_{\mathcal{B}}$.

 Given a univalent analytic function $f:\mathbb{D}\rightarrow \mathbb{C}$ we have that the function $\log f'$ is a Bloch function by the Koebe Distortion theorem with $\|\log f'\|_{\mathcal{B}}\leq 6$. Conversely, given a function $g\in\mathcal{B}$ with $\|g\|_{\mathcal{B}}\leq 2$, there exists a univalent function $f:\mathbb{D}\rightarrow\mathbb{C}$ such that $g=\log f'$, see (Chapter 4, \cite{P}). Given a conformal map $f:\mathbb{H}\rightarrow\mathbb{C}$, it follows by a conformal change of coordinates that
 $$
  \underset{z\in\mathbb{H}}\sup\; \text{Im}(z)|g'(z)|<6
 $$
where $g$ is the function $\log f'$.

Let us introduce some notation. Given a closed interval $I\subset \mathbb{R}$ we denote by $x_I$ the center of $I$ and $z_I:=x_I + i|I|$. We denote by $Q(I)$ the square $\{x+iy: x\in I,\,y\in(0,|I|)\}$. The intrinsic metric of a domain 
$\Omega\subset\mathbb{C}$ is given by 
$d_{\Omega}(x,y):=\inf \{l(\gamma_{x,y}): \gamma_{x,y} \;\text{is a rectifiable curve}\newline \text{joining}\; x\;\text{and}\; y\;\text{in}\;\Omega\}$. The euclidean disk with center $z$ and radius $r$ is denoted by $B(z,r)$ and 
$B_{\Omega}(z,r)$ is the corresponding intrinsic ball. 
%which is a ball in the intrinsic metric of the given domain $\Omega$ of the sam%e parameters is denoted $B_{\Omega}(z,r)$. 
We denote by diam($A$), the diameter of a set $A\subset\mathbb{C}$. We denote 
by $\text{diam}_{\Omega}(A)$ the diameter of a subset $A\subset\Omega$ 
measured with respect to the intrinsic metric of $\Omega$.
 
 \section{Proofs of the theorems}
We first sketch the proof of Theorem \ref{reg_boundary}. The set $E$ 
constructed below is a Cantor-type set. One considers the
harmonic function $u=\log|f'|$, the real part of the Bloch function $\log f'$, 
where $f$ is the conformal map from Theorem \ref{reg_boundary}. The construction involves selecting ``good''  parts in the boundary near which the function $u$ 
remains essentially bounded and estimating the size of the ``bad'' parts 
in the boundary where the difference from a fixed value is large and positive 
or large and negative. The good parts correspond to the points in the boundary, accessible from some interior point of $\Omega$ by a John curve. The key 
observation is that it is possible to recursively choose subsets from the bad 
parts of the  boundary, near which the difference from the fixed value 
is ``up'' and ``down'' at consecutive generations so that the final error in 
the intersection is not too large. The set $E$ consists of the good parts and an intersection of suitable nested sets of the bad part. The Hausdorff content of $E$ is shown to be large by the mass distribution principle after defining a limit measure supported on $E$. 

We use the following well known lemma in the proof of Theorem \ref{reg_boundary}; see Lemma 2.2 of \cite{N}. 

\begin{lem}\label{Green}
Let u be a harmonic function in the upper half plane $\mathbb{H}$ such that
\begin{equation*}
 \underset{z\in\mathbb{H}}\sup\; \text{Im}(z)|\nabla u(z)|\leq A.
\end{equation*}
 Let $I\subset\mathbb{R}$ be an interval and let $\{I_j\}$ be a collection of pairwise disjoint dyadic subintervals of $I$ and assume additionally that $u$ is bounded in $Q(I)\backslash \cup_j Q(I_j)$. Then we have
\begin{equation}\label{estimate_distortion}
u(z_I)=\underset{j}\sum u(z_{I_j})\frac{|I_j|}{|I|} + \frac{1}{|I|}\int_{I\backslash\cup_j I_j} u(x)dx + O(A).
\end{equation}
\end{lem}

\begin{proof}
 Let us write $y$ for the imaginary part of $z$ and fix $0<\epsilon<1$. Green's theorem applied to the harmonic functions $u$ and $y$ in the domain $U_{\epsilon}:=(Q(I)\backslash \overline{\cup_j(Q(I_j))}) \cap \{y>\epsilon\}$ gives
\begin{equation*}
 \int_{U_{\epsilon}} y\Delta u- \int_{\partial U_{\epsilon}} u\Delta y =  \int_{U_{\epsilon}} y \nabla u\cdot\nu ds - \int_{\partial U_{\epsilon}} u\nabla y\cdot \nu ds
\end{equation*}
and thus
\begin{equation}\label{greeneq}
 \int_{\partial U_{\epsilon}} u\nabla y\cdot \nu ds = \int_{\partial U_{\epsilon}} y \nabla u\cdot\nu ds
\end{equation}
where $\nu$ is the outward unit normal vector. The absolute value of the latter integral is bounded by $10A|I|$ by assumption. Note that the  oscillation of $u$ on the upper edges of $Q(I)$ and $Q(I_j)$ is bounded; indeed 
$$
 |u(x+i|I_j|)-u(x'+i|I_j|)|\leq A|x-x'|/|I_j| 
$$
for $x,x'\in I_j$. From (\ref{greeneq}) we have 
\begin{equation*}
 u(z_I)=\sum_{|I_j|>\epsilon} u(z_{I_j})\frac{|I_j|}{|I|} + \frac{1}{|I|}\int_{I} u(x+i\epsilon)\ind_{I\backslash (\underset{|I_j|>\epsilon}\cup I_j)}(x)\,dx + O(A)
\end{equation*}because the vertical sides of $\partial (Q(I_j))$ do not contribute to the integral.
The estimate now follows once we let $\epsilon\rightarrow 0$, since the function
$u$ has radial limits almost everywhere in $I\backslash \cup_j I_j.$ 
\end{proof}

\begin{lem}\label{size_intervals} Let $u$ be a harmonic function in the upper half plane $\mathbb{H}$ such that
\begin{equation*}
 \underset{z\in\mathbb{H}}\sup\; Im(z)|\nabla u(z)|\leq A.
\end{equation*} Then there is a number $M_0=M_0(A)$ such that the following holds for any $M>M_0$. 

Given any interval $I\subset\mathbb{R}$, define
\begin{equation*}
G(I):=\{\text{Re}(z): z\in Q(I),\,  \underset{0<\text{Im}(z)<|I|}\sup |u(z)-u(z_I)|\leq M+A\sqrt{2}\},
\end{equation*}
and assume that $|G(I)|\leq \frac{|I|}{100}$. Consider the family $\mathcal{F}(I)$ of maximal dyadic subintervals $I_j\subset I$ such that $|u(z_{I_j})-u(z_I)|\geq M$. Then we have
\begin{enumerate}
\item[(a)] $|u(z)-u(z_I)|\leq M+A\sqrt{2}$ for any $z\in Q(I)\backslash \underset{j}\cup Q(I_j)$.  In particular, $|u(z_{I_j})-u(z_I)|\leq M+A\sqrt{2}$.
\item[(b)] $|I_j|\leq 2^{-\frac{M}{A\sqrt{2}}} |I|$ for every $I_j\in \mathcal{F}(I)$.
\item[(c)] Consider the family $\mathcal{F}^+(I)$ (respectively $\mathcal{F}^-(I)$) of intervals in $\mathcal{F}(I)$ such that $u(z_{I_j})-u(z_I)\geq M$ (respectively $u(z_{I_j})-u(z_I)\leq -M$). Then we have 
\begin{enumerate}
 \item[(i)]$\underset{I_j\in \mathcal{F}^+(I)}\sum |I_j|\geq |I|/4$
 \item[(ii)]$\underset{I_j\in \mathcal{F}^-(I)}\sum |I_j|\geq |I|/4$
\end{enumerate}
\end{enumerate}
\end{lem}
\begin{proof}
 Given $z=x+iy\in \mathbb{H}$ such that $x\in I$ (respectively $I_j$) and $|I|/2<y<|I|$ (respectively $|I_j|/2<y<|I_j|$) it follows that $|u(z)-u(z_I)|\leq A\sqrt{2}$ (respectively $|u(z)-u(z_{I_j})|\leq A\sqrt{2}$) by our hypothesis. 
 
 Part (b) follows by iterating the above inequality and part (a) follows from the maximality of the dyadic intervals.
 
 For part (c) we write the estimate from Lemma \ref{Green} as
 \begin{equation*}
  \underset{j}\sum (u(z_{I_j})-u(z_I))\frac{|I_j|}{|I|} + \frac{1}{|I|} \underset{I\backslash \underset{j}\cup I_j}\int (u(x)-u(z_I))dx=\delta
 \end{equation*}
where $\delta=\delta(u,A)$ lies in the interval $[-\delta_A,\delta_A]$, where $\delta_A$ is a constant that depends only on $A$. We observe that $I\backslash \underset{j} \cup I_j\subset G(I)$. Thus the absolute value of the integral is bounded by $\frac{M+A\sqrt{2}}{100}$, by part (a) and the assumption that $|G(I)|\leq |I|/100$. Hence we have 
\begin{equation*}
 \abs{\underset{j}\sum (u(z_{I_j})-u(z_I))\frac{|I_j|}{|I|}}\leq \frac{M+A\sqrt{2}}{100}+ |\delta|.
\end{equation*}
Next we note that $M\leq |u(z_{I_j})-u(z_I)|\leq M+A\sqrt{2}$ for any $j$. Part (c) then follows from this. Indeed, if \begin{equation}\label{contra_1}\underset{I_j\in\mathcal{F}^+(I)}\sum \frac{|I_j|}{|I|}\leq \frac{1}{4},\end{equation} then we have
\begin{equation*}
\underset{I_j\in\mathcal{F}^-(I)}\sum \frac{|I_j|}{|I|}\geq \frac{74}{100}
\end{equation*}
and
\begin{equation*}
 -\frac{M+A\sqrt{2}}{100} - |\delta| \leq \frac{M+A\sqrt{2}}{4} + \underset{I_j\in\mathcal{F}^-(I)}\sum (u(z_{I_j})-u(z_I))\frac{|I_j|}{|I|}
\end{equation*}
from which we get
\begin{equation*}
 \underset{I_j\in\mathcal{F}^-(I)}\sum \frac{|I_j|}{|I|}\leq \frac{26}{100}+\frac{26A\sqrt{2}}{4M}+\frac{\delta_A}{M}.
\end{equation*}
This contradicts (\ref{contra_1}) if $M>M_0(A)$. The other inequality in part (c) follows similarly.
\end{proof}

Now we are ready to prove Theorem \ref{reg_boundary}.
\begin{proof}[Proof of Theorem \ref{reg_boundary}] Let $0<\alpha<1$ be fixed.

We may assume without loss of generality that $z_0=i$. We construct the set $E$ as follows. Set $u(z)=\log(|f'(z)|)$. Then $u$ is the real part of a Bloch function and thus satisfies the hypothesis of Lemma \ref{size_intervals} with $A=6$.

Denote by $I_0$ the interval $(-\frac{1}{2},\frac{1}{2})$. Consider the set $Q(I^{(0)})$ and the subset $G(I^{(0)})$ as defined in Lemma \ref{size_intervals}. An interval $I$ is called ``good'' if $|G(I)|\geq |I|/100$ and ``bad'' otherwise. If $|G(I^{(0)})|\geq |I^{(0)}|/100$, then set $E=G(I^{(0)})$. Then $\mathcal{H}^{\alpha}_{\infty}(E)\gtrsim 1$ and the claim follows.

So we assume that the other case holds and consider the maximal family $\mathcal{F}(I^{(0)})$ of subintervals $I_j\subset I_0$ as chosen in Lemma \ref{size_intervals}, with $M=M(\alpha)$ to be fixed later. Thus $I_0$ is a bad interval and we may apply Lemma \ref{size_intervals}. We have
\begin{equation*}
 |I|\leq 2^{-\frac{M}{6\sqrt{2}}} |I^{(0)}|\quad \text{if}\; I\in\mathcal{F}^+(I^{(0)})
\end{equation*}
and
$$
 \sum_{I\in\mathcal{F}^+(I^{(0)})}|I| \geq |I^{(0)}|/4.
$$

The first generation $\mathcal{G}_1=\mathcal{G}_1(I^{(0)})$ is formed by the subsets $G(I)$ of the good intervals $I\in\mathcal{F}^+(I^{(0)})$ and by the bad intervals $I\in\mathcal{F}^+(I^{(0)})$. We write $\mathcal{G}_1=\mathcal{G}^g_1\cup \mathcal{G}^b_1$ where 
\begin{equation*}
\mathcal{G}^g_1(I^{(0)})=\{G(I):I\in\mathcal{F}^+(I^{(0)})\;\text{is good}\}
\end{equation*}
and
\begin{equation*}
\mathcal{G}^b_1(I^{(0)})=\{I\in \mathcal{F}^+(I^{(0)}): I \; \text{is bad}\}.
\end{equation*}
We also have 
\begin{equation*}
 \sum_{I\in\mathcal{G}_1}|I|\geq |I^{(0)}|/400.
\end{equation*}

The construction stops in the sets in the family $\mathcal{G}_1^g$ of good sets. In the sets $I\in\mathcal{G}_1^b$ it continues as follows. 

Fix $I^{(1)}\in\mathcal{G}^b_1$. Since $I^{(1)}$ is bad we can apply Lemma \ref{size_intervals} and consider the collection $\mathcal{F}^-(I^{(1)})$ which satisfies 
\begin{equation*}
|I|\leq 2^{-\frac{M}{6\sqrt{2}}} |I^{(1)}|\quad \text{if}\; I\in\mathcal{F}^-(I^{(1)})
\end{equation*}
and
\begin{equation*}
\sum_{I\in\mathcal{F}^-(I^{(1)})}|I| \geq |I^{(1)}|/400.
\end{equation*}
The first generation $\mathcal{G}_1(I^{(1)})$ of the interval $I^{(1)}\in\mathcal{G}^b_1$ is written $\mathcal{G}^g_1(I^{(1)})=\mathcal{G}_1(I^{(1)})\cup \mathcal{G}^b_1(I^{(1)})$, where 
\begin{equation*}
 \mathcal{G}^g_1(I^{(1)})=\{G(I):I\in\mathcal{F}^-(I^{(1)})\;\text{is good}\}
\end{equation*}
and
\begin{equation*}
 \mathcal{G}^b_1(I^{(1)})=\{I\in\mathcal{F}^-(I^{(1)}):I\;\text{is bad}\}.
\end{equation*}

We use the first generation as defined above, of members of the collection $\mathcal{G}^b_1(I^{(0)}),$ to define the second generation $\mathcal{G}_2(I^{(0)})=\mathcal{G}^g_2(I^{(0)})\cup \mathcal{G}^b_2(I^{(0)})$, where
\begin{equation*}
 \mathcal{G}^g_2(I^{(0)})=\underset{I^{(1)}\in\mathcal{G}^b_1(I^{(0)})}\cup \mathcal{G}^g_1(I^{(1)})
\end{equation*}
and
\begin{equation*}
 \mathcal{G}^b_2(I^{(0)})=\underset{I^{(1)}\in\mathcal{G}^b_1(I^{(0)})}\cup \mathcal{G}^b_1(I^{(1)}).
\end{equation*}
We also have 
\begin{equation*}
 \sum_{\substack{I\subset I^{(1)}\\I\in\mathcal{G}_2}}|I|\geq |I^{(1)}|/400\quad \text{for any}\; I^{(1)}\in\mathcal{G}^b_1.
\end{equation*}

Observe that $u$ oscillates to the right of $u(z_{I_0})$ in the first step of the construction and to the left of $u(z_I)$ in the second. We have 
\begin{equation*}
 |u(z_I)-u(z_{I^{(0)}})|\leq 12\sqrt{2}\quad \text{for any}\; I\in\mathcal{G}^b_2.
\end{equation*}

Again the construction continues in the intervals of $\mathcal{G}^b_2(I^{(0)})$. Since the errors cancel but do not vanish, we use a slightly different value of $M$, if needed, for choosing the maximal family $\mathcal{F}(I)$ for the bad intervals $I\in\mathcal{G}^b_2$, so that the errors do not add up. More precisely, given $I^{(2)}\in\mathcal{G}^b_2$, we choose a value $M'$ from the interval $[M-6\sqrt{2},M+6\sqrt{2}]$ such that $u(z_{I^{(2)}})+M'=u(z_{I^{(0)}})+M$. 

So the construction stops after finitely many steps or continues indefinitely providing new generations $\mathcal{G}_n$. Let $I^{(n)}\in\mathcal{G}_n$. Either $I^{(n)}$ is of the form $G(\tilde{I})$ and the construction stops in $I^{(n)}$ or $I^{(n)}\in\mathcal{G}^b_n(I^{(0)})$ and the construction provides new sets and intervals of $\mathcal{G}_{n+1}$ contained in $I^{(n)}$ which satisfy 
\begin{equation*}
|I|\leq 2^{-\frac{M}{6\sqrt{2}}} |I^{(n)}|\quad \text{if}\; I\subset I^{(n)}\;,I\in\mathcal{G}_{n+1}
\end{equation*}
and
\begin{equation*}
\sum_{\substack{I\subset I^{(n)}\\I\in\mathcal{G}_{n+1}}}|I| \geq |I^{(n)}|/400.
\end{equation*}

We define 
$$
E:= (\cup_{n}\mathcal{G}^g_n) \cup (\cap_n \mathcal{G}^b_n).$$
By construction for any $x\in E$ and any $0\leq y \leq 1$ we have 
$$
 |\log|f'(x+iy)|-\log|f'(i)||\leq 2M+24.
$$
Consider the set $S(E)$ as defined in the statement of the theorem. Since $u$ is a Bloch function, the previous estimate gives that 
\begin{equation*}
e^{-2M-30}|f'(i)|\leq |f'(w)|\leq e^{2M+30}|f'(i)|
\end{equation*}
for any $w\in S(E)$. Thus part (b) of the statement follows.

Part (a) of the statement follows if it is shown that 
\begin{equation*}
\mathcal{H}^{\alpha}_{\infty}(E)\geq c(\alpha).
\end{equation*}
To prove the last estimate, it suffices by the mass distribution principle to construct a positive measure $\mu$ with $\mu(E)\geq 1$ such that there exists a constant $c(\alpha)>0$ with
$$
 \mu(I)\leq c(\alpha) |I|^{\alpha},
$$ for any interval $I\subseteq I_0$. 
The measure $\mu$ will be the limit of certain measures $\mu_n$ supported in the union $(\underset{k\leq n}\cup\mathcal{G}^g_k)\cup \mathcal{G}^b_n$, where $\mathcal{G}^g_k$ are the good parts of the previous generations.

Next we construct the measure $\mu$. 
Let $\mu_0=dx\restr I^{(0)}$. Consider \begin{equation*}a(I^{(0)})=\frac{|I^{(0)}|}{\sum_{I\in\mathcal{G}_1}|I|}\end{equation*}
which satisfies $a(I^{(0)})\leq 400$. By defining
$$
 \mu_1 := a(I^{(0)})\sum_{I\in\mathcal{G}_1} dx\restr I
$$
we have $\mu_1(I^{(0)})=1$. The measure $\mu_2$ will coincide with $\mu_1$ on $\mathcal{G}^g_1$. On $\mathcal{G}_2$ the measure $\mu_2$ will be defined by redistributing the mass of $\mu_1$. More concretely, if $I^{(1)}\in\mathcal{G}^b_1$ set 
\begin{equation*}
 a(I^{(1)})=\frac{\mu_1(I^{(1)})}{\sum_{\substack{I\subset I^{(1)}\\I\in\mathcal{G}_2}}|I|}.
\end{equation*}
Since 
\begin{equation*}
 a(I^{(1)})= \frac{|I^{(1)}|}{\sum_{\substack{I\subset I^{(1)}\\I\in\mathcal{G}_2}}|I|}\frac{\mu_1(I^{(1)})}{|I^{(1)}|}=a(I^{(0)})\frac{|I^{(1)}|}{\sum_{\substack{I\subset I^{(1)}\\I\in\mathcal{G}_2}}|I|},
\end{equation*}
we deduce that $a(I^{(1)})\leq 400^2$. Define 
\begin{equation*}
 \mu_2 = \mu_1\restr \mathcal{G}^g_1 + \sum_{I^{(1)}\in\mathcal{G}^b_1}a(I^{(1)})\sum_{\substack{I\subset I^{(1)}\\I\in\mathcal{G}_2}}dx\restr I.
\end{equation*}
The measures $\mu_3,\ldots,\mu_n,\ldots$ are defined recursively. Observe that $\mu_k(I)=\mu_n(I)$ for any $k\geq n$, provided $I\in\mathcal{G}_n$. Moreover, if $I\in\mathcal{G}_n$ we have \begin{equation*}\frac{\mu_n(I)}{|I|}\leq 400^n.\end{equation*}

Finally set \begin{equation*}\mu=\underset{n\rightarrow\infty}\lim \mu_n.\end{equation*}

It is clear that spt$\mu\subset E$ and $\mu(E)=1$. We want to check that $\mu(I)\leq c(\alpha)|I|^{\alpha}$ for any interval $I\subseteq I_0$. Let $J\subset I_0$ be an interval. We may assume that there is a positive integer $j$ such that \begin{equation*}2^{-\frac{M(j+1)}{6\sqrt{2}}}\leq |J|\leq 2^{-\frac{Mj}{6\sqrt{2}}}.\end{equation*} Let $\mathcal{G}_j(J)$ (respectively $\mathcal{G}^g_k(J)$) be the family of sets of generation $\mathcal{G}_j$ (respectively $\mathcal{G}^g_k$) which intersect $J$. Let $\mathcal{A}_j(J)$ be the family  of sets in $\cup^{j-1}_{k=0}\mathcal{G}^g_k(J)$ of diameters smaller than $2^{-\frac{Mj}{6\sqrt{2}}}$. Since the sets in $\mathcal{G}_j(J)\cup \mathcal{A}_j(J)$ intersect $J$ and have diameter smaller than $2|J|$, we have 
$$
 \mathcal{G}_j(J)\cup\mathcal{A}_j(J)\subset 4J.
$$
Hence
\begin{equation*}
\begin{split}
 \mu(J)&\leq \sum_{I\in\mathcal{G}_j(J)}\mu(I) + \sum^{j-1}_{k=0}\sum_{I\in\mathcal{G}^g_k(J)}\mu(I\cap J)\\
 &\leq \sum_{I\in\mathcal{G}_j(J)\cup \mathcal{A}_j(J)}\mu(I) + \sum_{I\in\cup^{j-1}_{k=0}\mathcal{G}^g_k(J)\backslash\mathcal{A}_j(J)}\mu(I\cap J)\\
 &=: A+B.
 \end{split}
\end{equation*}

If $I\in\mathcal{G}_j(J)\cup \mathcal{A}_j(J)$, then we have 
\begin{equation*}
 \mu(I)=\mu_j(I)=\frac{\mu_j(I)}{|I|}|I|\leq 400^j |I|.
\end{equation*}
Hence
\begin{equation*}
 A\leq 400^j \sum_{\substack{I\subset 4J\\I\in\mathcal{G}_j(J)\cup\mathcal{A}_j(J)}}|I|\leq 4\cdot 400^j|J|.
\end{equation*}
Since the sets in $\cup^{j-1}_{k=0}\mathcal{G}^g_k(J)\backslash \mathcal{A}_j(J)$ intersect $J$ and are contained in intervals of length larger than $|J|$ which are pairwise disjoint, the collection $\cup^{j-1}_{k=0}\mathcal{G}^g_k(J)\backslash \mathcal{A}_j(J)$ is contained in at most two intervals $L_1$ and $L_2$. Now
\begin{equation*}
 \begin{split}
 B&\leq \sum^2_{i=1} \mu(L_i)\leq 400^j \sum^2_{i=1} |L_i\cap J|\leq 400^j[J| .
 \end{split}
\end{equation*}
Hence $\mu(J)\leq A+B\leq 5\cdot 400^j |J|$. 

Since $|J|\leq 2^{-\frac{Mj}{6\sqrt{2}}}$, we deduce that 
$$
 \mu(J)\leq 5|J|^{1-\frac{6\sqrt{2}\ln_2(400)}{M}}.
$$
We choose $M$ large enough so that $1-\frac{6\sqrt{2}\ln_2(400)}{M}>\alpha$. The theorem follows with this value of $M$.
\end{proof}
The proof of Theorem \ref{main_reg} uses the following auxiliary result.
\begin{lem}\label{imagecontentbound}
Assume that a set $E=E(z_0,\alpha)\subset\mathbb{R}$ exists, for which part (b) of Theorem \ref{reg_boundary} is satisfied. Then we have that $f(S(E))$ is a John domain with John constant depending on $\alpha$. Moreover, 
\begin{equation*}
\mathcal{H}^{\alpha}_{\infty}(f(E))\geq c(\alpha)|f'(z_0)|^{\alpha} \mathcal{H}^{\alpha}_{\infty}(E).
\end{equation*}
\end{lem}
\begin{proof}
We may assume that $E$ is compact. We have for $f(z)\in f(S(E))$ that 
\begin{equation*} 
\text{d}_{f(S(E))}(f(z),\partial f(S(E)))\gtrsim c(\alpha)|f'(z_0)|\text{d}_{S(E)}(z)
\end{equation*}
from which it follows that $f(S(E))$ is John, with John curves being the images of the John curves in $S(E)$ (which we may take to be the vertical line segment joining the point $z$ to the upper edge of $Q(I)$ followed by a horizontal line segment till $z_I$). The constant only depends on $\alpha$.

Let $\{B_i\}_{i}$ be a countable (possibly finite) collection of disks covering $f(E)$. We may assume that $f(z_0)\notin B_i$ for all indices $i$.  For each point $f(\hat{z})\in f(E)\cap B_i$ consider the John curve $\gamma_{\hat{z}}$ joining $f(\hat{z})$ to  the John center $f(z_0)$. Let $f(z)\in \gamma_{\hat{z}}$ be a point such that $l({\gamma_{\hat{z}}}(f(\hat{z}),f(z)))=\text{rad}(B_i)=:R_i$. By the John condition we have that $\text{d}_{f(S(E))}(f(z))\geq \frac{1}{c(\alpha)} R_i$. 

In the following we write $S'(E)$ for the set $f(S(E))$. Consider the collection of intrinsic balls $\{B_{S'(E)}(f(z),c(\alpha)\text{d}_{S'(E)}(f(z)))\}_{f(\hat{z})}$ in the intrinsic metric  of $S'(E)$. The balls in this collection cover the set $f(E)\cap B_i$. By the $5r$-covering theorem we find pairwise disjoint balls \begin{equation*}B_{S'(E)}(f(z^i_j),c(\alpha)\text{d}_{S'(E)}(f(z^i_j))),\,j=1,2,\ldots\end{equation*} such that $\{B_{S'(E)}(f(z^i_j),5c(\alpha)\text{d}_{S'(E)}(f(z^i_j)))\}_j$ covers $f(E)\cap B_i$.

 We have also an upper bound $N(\alpha)$ for the number $N_i$ of the pairwise disjoint intrinsic balls found above for each $f(E)\cap B_i$, since every ball $B_{S'(E)}(f(z^i_j),c(\alpha)\text{d}_{S'(E)}(f(z^i_j)))$ in the collection contains the euclidean disk $B(f(z^i_j),R_i/c(\alpha))$. Let the collection of finitely many such intrinsic balls chosen for each index $i$ be denoted together $\{B_{ij}\}_{\substack{i\in\mathbb{N}\\1\leq j\leq N_i}}$ where for given $i$ and $1\leq j\leq N_i$, $B_{ij}=B_{S'(E)}(f(z^i_j),5c(\alpha)\text{d}_{S'(E)}(f(z^i_j)))$. We have
\begin{equation*}
\begin{split}
 \underset{i}\sum (\text{diam}(B_i))^{\alpha} &\gtrsim \underset{i}\sum \sum_{j=1}^{N_i}(\text{diam}_{S'(E)}(B_{ij}))^{\alpha}\\
 &\geq c(\alpha) |f'(z_0)|^{\alpha}\sum_{i,j}(\text{diam}_{S(E)} (f^{-1}(B_{ij})))^{\alpha}\\
 & \geq c(\alpha)|f'(z_0)|^{\alpha}\mathcal{H}_{\infty}^{\alpha}(E)
 \end{split}
\end{equation*}
The lemma follows.
\end{proof}
\begin{proof}[Proof of Theorem \ref{main_reg}] Let $f$ be the conformal map from $\mathbb{H}$ to $\Omega$. Consider the set $E=E(f^{-1}(z),\alpha)$ obtained using Theorem \ref{reg_boundary}. Applying Lemma \ref{imagecontentbound} we have 
\begin{equation*}
 \mathcal{H}^{\alpha}_{\infty}(f(E))\gtrsim (e^{-2M}|f'(f^{-1}(z))|)^{\alpha}\mathcal{H}^{\alpha}_{\infty}(E).
\end{equation*}
Theorem \ref{main_reg} now follows by combining the above estimate with part (a) of Theorem \ref{reg_boundary} and observing that, by part (b), $z$ can be joined to $\partial\Omega$ by a curve which is the bilipschitz image of a curve in $\mathbb{H}$ of length comparable to $\text{Im}\,f^{-1}(z)$ joining $f^{-1}(z)$ to $\mathbb{R}$.

\end{proof}

\end{document}